\documentclass[11pt,a4paper,reqno]{amsart}
\usepackage[british]{babel}

\usepackage{a4wide}
\usepackage{setspace}
\usepackage{amssymb}
\usepackage{amsmath}
\usepackage{amsthm}
\usepackage{enumitem}
\usepackage{mathrsfs}
\usepackage[colorlinks,citecolor=blue,urlcolor=black,linkcolor=red,linktocpage]{hyperref}
\usepackage{mathrsfs}
\usepackage{mathtools}

\newtheorem{thm}{Theorem}[section]

\newtheorem{cor}[thm]{Corollary}
\newtheorem{lem}[thm]{Lemma}

\newtheorem{prop}[thm]{Proposition}
\newtheorem{problem}[thm]{Problem}

\theoremstyle{definition}

\newtheorem{definition}[thm]{Definition}

\newtheorem{remark}[thm]{Remark}
\newtheorem{remarks}[thm]{Remarks}

\renewcommand{\epsilon}{\varepsilon}
\renewcommand{\phi}{\varphi}
\newcommand{\defeq}{\mathrel{\mathop:}=}

\DeclareMathOperator{\spt}{spt}
\DeclareMathOperator{\diam}{diam}

\makeatletter
\def\moverlay{\mathpalette\mov@rlay}
\def\mov@rlay#1#2{\leavevmode\vtop{%
		\baselineskip\z@skip \lineskiplimit-\maxdimen
		\ialign{\hfil$\m@th#1##$\hfil\cr#2\crcr}}}
\newcommand{\charfusion}[3][\mathord]{
	#1{\ifx#1\mathop\vphantom{#2}\fi
		\mathpalette\mov@rlay{#2\cr#3}
	}
	\ifx#1\mathop\expandafter\displaylimits\fi}
\makeatother

\newcommand{\bigcupdot}{\charfusion[\mathop]{\bigcup}{\cdot}}

\begin{document}

\setlist{noitemsep}

\author{Friedrich Martin Schneider}
\address{F.M.S., Institute of Algebra, TU Dresden, 01062 Dresden, Germany }
\email{martin.schneider@tu-dresden.de}

\title[Equivariant dissipation]{Equivariant dissipation in non-archimedean groups}
\date{\today}


\begin{abstract} 
 We prove that, if a topological group $G$ has an open subgroup of infinite index, then every net of tight Borel probability measures on $G$ UEB-converging to invariance dissipates in $G$ in the sense of Gromov. In particular, this solves a 2006 problem by Pestov: for every left-invariant (or right-invariant) metric $d$ on the infinite symmetric group $\mathrm{Sym}(\mathbb{N})$, compatible with the topology of pointwise convergence, the sequence of the finite symmetric groups $(\mathrm{Sym}(n),d\!\!\upharpoonright_{\mathrm{Sym}(n)},\mu_{\mathrm{Sym}(n)})_{n \in \mathbb{N}}$ equipped with the restricted metrics and their normalized counting measures dissipates, thus fails to admit a subsequence being Cauchy with respect to Gromov's observable distance.
\end{abstract}


\maketitle


\section{Introduction}

In his seminal work on metric measure geometry~\cite[Chapter~3$\tfrac{1}{2}$]{Gromov99}, Gromov introduced the \emph{observable distance}, $d_{\mathrm{conc}}$, a metric on the set of isomorphism classes of $mm$-spaces, i.e., separable complete metric spaces equipped with a Borel probability measure. This metric generates an interesting topology, commonly referred to as the \emph{concentration topology}, which generalizes the \emph{L\'evy concentration property}~\cite{levy,milman,MilmanSchechtman} in a very natural way: a sequence of $mm$-spaces constitutes a L\'evy family if and only if it converges to a singleton space with respect Gromov's concentration topology. Inspired by the work of Gromov and Milman~\cite{GromovMilman} on applications of concentration to dynamics of topological groups, Pestov proposed to study instances of concentration to non-trivial spaces in the context of topological groups~\cite[Section~7.4]{PestovBook} (see also his work with Giordano~\cite[Section~7]{GiordanoPestov}).

In the present note, we study Gromov's observable distance with regard to \emph{non-archimedean} topological groups, i.e., those whose neutral element admits a neighborhood basis consisting of (open) subgroups. More specifically, our focus will be on the topological group $\mathrm{Sym}(\mathbb{N})$ of all permutations of the set $\mathbb{N}$ of natural numbers, endowed with the topology of pointwise convergence. In~\cite{GlasnerWeiss02}, Glasner and Weiss showed that the closed subspace $\mathrm{LO}(\mathbb{N}) \subseteq 2^{\mathbb{N} \times \mathbb{N}}$ of linear orders on $\mathbb{N}$, equipped with the natural continuous left $\mathrm{Sym}(\mathbb{N})$-action given by \begin{displaymath}
	x \, \, {{}^{g}\!\!\prec} \, \, y \ \, \Longleftrightarrow \ \, g^{-1}x \prec g^{-1}y \qquad (g \in \mathrm{Sym}(\mathbb{N}), \, {\prec} \in \mathrm{LO}(\mathbb{N}), \, x,y \in \mathbb{N}) ,
\end{displaymath} constitutes the universal minimal flow of $\mathrm{Sym}(\mathbb{N})$ and admits a unique $\mathrm{Sym}(\mathbb{N})$-invariant Borel probability measure $\mu_{\mathrm{LO}}$. Their results prompted Pestov to pose the following question.

\begin{problem}[\cite{PestovBook}, Problem~7.4.27]\label{problem} For each $n \in \mathbb{N}$, let us denote by $\mu_{n}$ the normalized counting measure on $\mathrm{Sym}(n) \subseteq \mathrm{Sym}(\mathbb{N})$. Do there exist compatible metrics $d_{\mathrm{LO}}$ on $\mathrm{LO}(\mathbb{N})$ and $d_{\mathrm{Sym}}$, left-invariant, on $\mathrm{Sym}(\mathbb{N})$ such that \begin{displaymath}
	d_{\mathrm{conc}} \bigl( \bigl(\mathrm{Sym}(n),d_{\mathrm{Sym}}\!\!\upharpoonright_{\mathrm{Sym}(n)},\mu_{n}\bigr),\bigl(\mathrm{LO}(\mathbb{N}),d_{\mathrm{LO}},\mu_{\mathrm{LO}}\bigr)\bigr) \, \longrightarrow \, 0 \quad (n \longrightarrow \infty) \ ?
\end{displaymath}  \end{problem}

The purpose of this note is to resolve Problem~\ref{problem} in the negative. In fact, our Corollary~\ref{corollary:final} particularly entails that, if $d$ is any left-invariant compatible metric on $\mathrm{Sym}(\mathbb{N})$, then the sequence $\bigl(\mathrm{Sym}(n),d\!\!\upharpoonright_{\mathrm{Sym}(n)},\mu_{n}\bigr)_{n \in \mathbb{N}}$ does not even admit a $d_{\mathrm{conc}}$-Cauchy subsequence, where $\mu_{n}$ denotes the normalized counting measure on $\mathrm{Sym}(n)$ for each $n \in \mathbb{N}$. Our argument proving Corollary~\ref{corollary:final} does indeed establish \emph{dissipation} (Definition~\ref{definition:dissipation}), a phenomenon stronger than the negation of concentration (cf.~Corollary~\ref{corollary:dissipation.vs.concentration}, Remark~\ref{remark:dissipation.vs.concentration}), and moreover works in much greater generality, thus allowing us to deduce a dichotomy concerning the geometric behavior of asymptotically invariant nets of Borel probability measures in arbitrary non-archimedean topological groups (Theorem~\ref{theorem:equivariant.dissipation}, Corollary~\ref{corollary:equivariant.dissipation}, Corollary~\ref{corollary:dichotomy}). 

This note is organized as follows. In Section~\ref{section:metrics} we recollect some basic material from metric geometry, most importantly, the Gromov-Hausdorff distance and Gromov's compactness theorem. Then, Section~\ref{section:measures} is devoted to a short introduction to $mm$-spaces and observable distance, as well as a brief comparison of the concepts of concentration and dissipation. In Section~\ref{section:invariance} we provide the background on UEB-convergence to invariance in topological groups necessary for Section~\ref{section:groups}, where we turn our attention towards non-archimedean topological groups, prove the aforementioned dichotomy, and infer the solution to Pestov's Problem~\ref{problem}.

\section{Metric geometry: Gromov's compactness theorem}\label{section:metrics}

In this section, we recollect some very few bits of metric geometry, the most important of which will be the Gromov-Hausdorff distance and Gromov's compactness theorem. For more on this, the reader is referred to~\cite{BBI} or~\cite[Chapter~3]{ShioyaBook}.

For a start, let us briefly clarify some basic notation and terminology concerning metric spaces. By a \emph{compatible} metric on a (metrizable) topological space $X$, we will mean a metric generating the topology of $X$. Let $\mathcal{X} = (X,d)$ be a pseudo-metric space. The \emph{diameter} of $\mathcal{X}$ is defined as $\diam (\mathcal{X}) \defeq \sup \{ d(x,y) \mid x,y \in X \}$. Given any real number $\ell \geq 0$, let us denote by $\mathrm{Lip}_{\ell}(\mathcal{X})$ the set of all $\ell$-Lipschitz real-valued functions on $\mathcal{X}$, and define \begin{displaymath}
	\mathrm{Lip}_{\ell}^{s}(\mathcal{X}) \defeq \{ f \in \mathrm{Lip}_{\ell}(\mathcal{X}) \mid \sup\nolimits_{x \in X} \vert f(x) \vert \leq s \}
\end{displaymath} for any real number $s \geq 0$. For a real number $\epsilon > 0$, a subset $B \subseteq X$ is said to be \emph{$\epsilon$-discrete} in $\mathcal{X}$ if $d(x,y) > \epsilon$ for any two distinct $x,y \in B$, and the \emph{$\epsilon$-capacity} of $\mathcal{X}$ is defined as \begin{displaymath}
	\mathrm{Cap}_{\epsilon}(\mathcal{X}) \defeq \sup \{ \vert B \vert \mid B \subseteq X \text{ $\epsilon$-discrete in } \mathcal{X} \} .
\end{displaymath} Given a subset $A \subseteq X$, we abbreviate $d\!\!\upharpoonright_{A} \, \defeq d\vert_{A \times A}$. For $x \in A \subseteq X$ and $\epsilon > 0$, we let \begin{align*}
	& B_{d}(x,\epsilon) \defeq \{ y \in X \mid d(x,y) < \epsilon \} , & B_{d}(A,\epsilon) \defeq \{ y \in X \mid \exists a \in A \colon \, d(a,y) < \epsilon \} .
\end{align*} The \emph{Hausdorff distance} of two subsets $A,B \subseteq X$ in $\mathcal{X}$ is denoted by \begin{displaymath}
	\mathrm{H}_{\mathcal{X}}(A,B) \defeq \mathrm{H}_{d}(A,B) \defeq \inf \{ \epsilon > 0 \mid B \subseteq B_{d}(A,\epsilon), \, A \subseteq B_{d}(B,\epsilon) \} .
\end{displaymath}

\begin{definition} The \emph{Gromov-Hausdorff distance} between any two arbitrary compact metric spaces $\mathcal{X} = (X,d_{X})$ and $\mathcal{Y} = (Y,d_{Y})$ is defined as \begin{displaymath}
	\left. d_{\mathrm{GH}}(\mathcal{X},\mathcal{Y}) \defeq \inf \left\{ \mathrm{H}_{\mathcal{Z}}(\phi(X),\psi(Y)) \, \right| \mathcal{Z} \text{ metric space, } \phi \colon \mathcal{X} \to \mathcal{Z}, \, \psi \colon \mathcal{Y} \to \mathcal{Z} \text{ isom.~emb.} \right\} .
\end{displaymath} \end{definition}

The Gromov-Hausdorff distance of compact metric spaces is easily seen to be invariant under isometries, i.e., $d_{\mathrm{GH}}(\mathcal{X}_{0},\mathcal{X}_{1}) = d_{\mathrm{GH}}(\mathcal{Y}_{0},\mathcal{Y}_{1})$ for any two pairs of isometrically isomorphic compact metric spaces $\mathcal{X}_{i} \cong \mathcal{Y}_{i}$ ($i \in \{ 0,1 \}$). Furthermore, $d_{\mathrm{GH}}$ gives a complete metric on the set of isomorphism classes of compact metric spaces~\cite[Lemma~3.9]{ShioyaBook}. In particular, two compact metric spaces $\mathcal{X}$ and $\mathcal{Y}$ are isometrically isomorphic if and only if $d_{\mathrm{GH}}(\mathcal{X},\mathcal{Y}) = 0$.

Let us recall below a useful description of $d_{\mathrm{GH}}$-precompactness in terms of capacities, known as Gromov's compactness theorem. We will say that a set $\mathcal{C}$ of (isometry classes of) compact metric spaces has \emph{uniformly bounded capacity} if $\sup\nolimits_{\mathcal{X} \in \mathcal{C}} \mathrm{Cap}_{\epsilon}(\mathcal{X}) < \infty$ for every $\epsilon > 0$.
		
\begin{thm}[cf.~\cite{ShioyaBook}, Lemma~3.12; \cite{BBI}, Section~7.4.2]\label{theorem:gromov.precompactness} A set $\mathcal{C}$ of isometry classes of compact metric spaces is $d_{\mathrm{GH}}$-precompact if and only if $\mathcal{C}$ has uniformly bounded capacity and $\sup_{\mathcal{X} \in \mathcal{C}} \diam (\mathcal{X}) < \infty$. \end{thm}

\section{Metric measure geometry: concentration vs.~dissipation}\label{section:measures}
		
We now turn to the study of measured metric spaces. In this section, we will only briefly review the concepts of \emph{concentration} and \emph{dissipation}, two phenomena at opposite ends of the spectrum of the asymptotic behavior of measured metric spaces. A more substantial account on metric measure geometry is to be found in~\cite{Gromov99,ShioyaBook}. For more details on the classical phenomenon of measure concentration, the reader is referred to~\cite{MilmanSchechtman,ledoux}.

To begin with, let us address some few matters of notation. Let $X$ be a measurable~space. As usual, we will denote by $\delta_{x}$ the Dirac measure on $X$ associated with a point~$x \in X$. Furthermore, given a finite non-empty subset $F \subseteq X$, we will consider the probability measure $\delta_{F} \defeq \vert F \vert^{-1}\sum_{x \in F} \delta_{x}$ on $X$. We denote by $\mathrm{P}(X)$ the set of all probability measures on $X$. Consider any $\mu \in \mathrm{P}(X)$. If $B$ is a measurable subset of $X$ with $\mu (B) = 1$, then we may consider the probability measure $\mu \!\!\upharpoonright_{B}$ on the measurable subspace $B$ given by $\mu \!\!\upharpoonright_{B} \, \defeq \mu(A)$ for every measurable subset $A \subseteq B$. The \emph{push-forward measure} $f_{\ast}(\mu)$ of $\mu$ along a measurable map $f \colon X \to Y$ into another measurable space $Y$ is defined by $f_{\ast}(\mu)(B) \defeq \mu(f^{-1}(B))$ for every measurable $B \subseteq Y$. Moreover, we obtain a pseudo-metric $\mathrm{me}_{\mu}$ on the set of all measurable real-valued functions $X$ defined by \begin{displaymath}
	\mathrm{me}_{\mu}(f,g) \defeq \inf \{ \epsilon > 0 \mid \mu (\{ x \in X \mid \vert f(x) - g(x) \vert > \epsilon \}) \leq \epsilon \} 
\end{displaymath} for any two measurable functions $f,g \colon X \to \mathbb{R}$. Finally, let $\nu$ be a Borel probability measure on a Hausdorff topological space $T$. Then the \emph{support} of $\nu$ is defined as \begin{displaymath}
	\spt \nu \defeq \{ x \in T \mid \forall U \subseteq T \text{ open} \colon \, x \in U \Longrightarrow \nu (U) > 0\} ,
\end{displaymath} which forms a closed subset of $T$. Furthermore, $\nu$ will be called \emph{tight}~\cite[Definition~7.1.4]{Bogachev} if for every $\epsilon > 0$ there exists a compact subset $K \subseteq T$ with $\nu (K) \geq 1 -\epsilon$, and $\nu$ will be called \emph{regular} if $\nu (B) = \sup \{ \nu(K) \mid K\subseteq B, \, K \text{ compact} \}$ for every Borel subset $B \subseteq T$. Evidently, regularity implies tightness. Conversely, if $T$ is metrizable and $\nu$ is tight, then $\nu$ will be~regular as well (see, e.g.,~\cite[Chapter~II, Theorem~3.1]{part}). For additional measure-theoretic background, we refer to~\cite{part,Bogachev}.

\begin{definition} Let $\mathcal{X} = (X,d,\mu)$ be an \emph{$mm$-space}, that is, $(X,d)$ is a separable complete metric space and $\mu$ is a Borel probability measure on $X$. We will call $\mathcal{X}$ \emph{compact} if $(X,d)$ is compact, and \emph{fully supported} if $\spt \mu = X$. A \emph{parametrization} of $\mathcal{X}$ is a Borel measurable map $\phi \colon [0,1] \to X$ such that $\phi_{\ast}(\lambda) = \mu$, where $\lambda$ denotes the Lebesgue measure on~$[0,1]$. We will call two $mm$-spaces $\mathcal{X}_{0} = (X_{0},d_{0},\mu_{0})$ and $\mathcal{X}_{1} = (X_{1},d_{1},\mu_{1})$ \emph{isomorphic} and write~$\mathcal{X}_{0} \cong \mathcal{X}_{1}$ if there exists an isometry \begin{displaymath}
	f \colon \left(\spt \mu_{0},d_{0}\! \! \upharpoonright_{\spt \mu_{0}}\right) \, \longrightarrow \, \left(\spt \mu_{1},d_{1}\! \! \upharpoonright_{\spt \mu_{1}}\right)
\end{displaymath} such that $f_{\ast}\! \left(\mu_{0}\! \! \upharpoonright_{\spt \mu_{0}}\right) = \mu_{1}\! \! \upharpoonright_{\spt \mu_{1}}$. \end{definition}
		
It is well known that any $mm$-space admits a parametrization (see e.g.~\cite[Lemma~4.2]{ShioyaBook}).

\begin{definition}\label{definition:observable.distance} The \emph{observable distance} between two $mm$-spaces $\mathcal{X}$ and $\mathcal{Y}$ is defined to be \begin{displaymath}
	\left. d_{\mathrm{conc}}(\mathcal{X},\mathcal{Y}) \defeq \inf \left\{ \mathrm{H}_{\mathrm{me}_{\lambda}}(\mathrm{Lip}_{1}(\mathcal{X}) \circ \phi, \mathrm{Lip}_{1}(\mathcal{Y}) \circ \psi) \, \right| \phi \text{ param.~of } \mathcal{X}, \, \psi \text{ param.~of } \mathcal{Y} \right\} .
\end{displaymath} A sequence $(\mathcal{X}_{n})_{n \in \mathbb{N}}$ of $mm$-spaces is said to \emph{concentrate to} an $mm$-space $\mathcal{X}$ if \begin{displaymath}
	\lim\nolimits_{n \to \infty} d_{\mathrm{conc}}(\mathcal{X}_{n},\mathcal{X}) = 0 .
\end{displaymath} \end{definition}

It is easy to see that the observable distance of $mm$-spaces is invariant under isomorphisms, which means that $d_{\mathrm{conc}}(\mathcal{X}_{0},\mathcal{X}_{1}) = d_{\mathrm{conc}}(\mathcal{Y}_{0},\mathcal{Y}_{1})$ for any two pairs of isomorphic $mm$-spaces $\mathcal{X}_{i} \cong \mathcal{Y}_{i}$ ($i \in \{ 0,1 \}$). Furthermore, $d_{\mathrm{conc}}$ induces a metric on the set of isomorphism classes of \mbox{$mm$-spaces}, see~\cite[Theorem~5.16]{ShioyaBook}. In particular, this entails that two $mm$-spaces $\mathcal{X}$ and $\mathcal{Y}$ are isomorphic if and only if $d_{\mathrm{conc}}(\mathcal{X},\mathcal{Y}) = 0$.

Let us proceed to the concept of dissipation, cf.~\cite[Chapter~3$\tfrac{1}{2}$.J]{Gromov99}, \cite[Chapter~8]{ShioyaBook}. As will become evident in Corollary~\ref{corollary:dissipation.vs.concentration}, dissipating sequences of $mm$-spaces are, in a certain sense, as far from being convergent with respect to Gromov's observable distance as possible. The observation most crucial for the proof of Corollary~\ref{corollary:dissipation.vs.concentration} will be given in Lemma~\ref{lemma:capacity}.

\begin{definition}\label{definition:dissipation} Let $\mathcal{X} = (X,d,\mu)$ be an $mm$-space. For every $m \in \mathbb{N}\setminus \{ 0 \}$ and real numbers $\kappa_{0},\ldots,\kappa_{m} > 0$, the corresponding \emph{separation distance} is defined as \begin{displaymath}
	\mathrm{Sep}(\mathcal{X};\kappa_{0},\ldots,\kappa_{m}) \defeq \sup\nolimits_{B \in [\mathcal{X};\kappa_{0},\ldots,\kappa_{m}]} \inf \{ d(x,y) \mid i,j \in \{ 0,\ldots,m \}, \, i\ne j, \, x \in B_{i}, \, y \in B_{j} \} ,
\end{displaymath} where we abbreviate \begin{displaymath}
	[\mathcal{X};\kappa_{0},\ldots,\kappa_{m}] \defeq \{ (B_{0},\ldots,B_{m}) \mid B_{0},\ldots,B_{m} \subseteq X \text{ Borel}, \, \mu(B_{0}) \geq \kappa_{0}, \ldots, \mu(B_{m}) \geq \kappa_{m} \} 
\end{displaymath} and the infima and suprema are taken in the interval $[0,\infty)$.\footnote{For any $x \in X$ there exists $n \in \mathbb{N}\setminus \{ 0 \}$ such that $\mu (B_{d}(x,n)) > 1 - \min \{ \kappa_{0},\ldots,\kappa_{m} \}$, which is easily seen to imply that $\mathrm{Sep}(\mathcal{X};\kappa_{0},\ldots,\kappa_{m}) \leq 2n$. In particular, $\mathrm{Sep}(\mathcal{X};\kappa_{0},\ldots,\kappa_{m}) < \infty$.} For any $m \in \mathbb{N}\setminus \{ 0 \}$ and any real number $\alpha > 0$, we define $\mathrm{Sep}_{m}(\mathcal{X};\alpha) \defeq \mathrm{Sep}(\mathcal{X};\kappa_{0},\ldots,\kappa_{m})$ where $\kappa_{0} = \ldots = \kappa_{m} = \alpha$. With regard to a real number $\delta > 0$, a sequence of $mm$-spaces $(\mathcal{X}_{n})_{n \in \mathbb{N}}$ is said to \emph{$\delta$-dissipate} if, for any $m \in \mathbb{N}\setminus \{ 0 \}$ and real numbers $\kappa_{0},\ldots,\kappa_{m} > 0$ with $\sum_{i=0}^{m} \kappa_{i} < 1$, \begin{displaymath}
	\liminf\nolimits_{n \to \infty} \mathrm{Sep}(\mathcal{X}_{n};\kappa_{0},\ldots,\kappa_{m}) \, \geq \, \delta .
\end{displaymath} A sequence of $mm$-spaces is said to \emph{dissipate} if it $\delta$-dissipates for some $\delta > 0$. \end{definition}

Let us note the following useful reformulation of dissipation.

\begin{lem}\label{lemma:dissipation.revisited} Let $\delta > 0$. A sequence of $mm$-spaces $\mathcal{X}_{n} = (X_{n},d_{n},\mu_{n})$ $(n \in \mathbb{N})$ $\delta$-dissipates if and only if, for every $\tau \in (0,\delta)$, there exists a sequence $(\mathcal{B}_{n})_{n \in \mathbb{N}}$ such that \begin{enumerate}
	\item[$(0)$] \, for each $n \in \mathbb{N}$, $\mathcal{B}_{n}$ is a finite set of Borel subsets of $X_{n}$,
	\item[$(1)$] \, $\inf \{ d_{n}(x,y) \mid x \in B, \, y \in C \} \geq \tau$ for every $n \in \mathbb{N}$ and any two distinct $B,C \in \mathcal{B}_{n}$,
	\item[$(2)$] \, $\lim_{n \to \infty}\mu_{i}(\bigcup \mathcal{B}_{n}) = 1$, and
	\item[$(3)$] \, $\lim_{n \to \infty} \sup \{ \mu_{n}(B) \mid B \in \mathcal{B}_{n} \} = 0$.
\end{enumerate}\end{lem}

\begin{proof} ($\Longrightarrow$) Let $\tau \in (0,\delta)$. For each $n \in \mathbb{N}$, we define \begin{displaymath}
	k_{n} \defeq \sup \left\{ k \in \{ 1,\ldots, n \} \left\vert \, \mathrm{Sep}_{k-1}\!\left(\mathcal{X}_{n};(k+1)^{-1}\right) > \tau \right\} . \right.
\end{displaymath} Since $(\mathcal{X}_{n})_{n \in \mathbb{N}}$ $\delta$-dissipates, $k_{n} \longrightarrow \infty$ as $n \longrightarrow \infty$. For each $n \in \mathbb{N}$, let us choose Borel subsets $B_{n,0},\ldots,B_{n,k_{n}-1} \subseteq X_{n}$ such that \begin{enumerate}
	\item[$\bullet$] \, $\inf \{ d_{n}(x,y) \mid x \in B_{n,i}, \, y \in B_{n,j} \} \geq \tau$ for any two distinct $i,j \in \{ 0,\ldots,k_{n}-1 \}$,
	\item[$\bullet$] \, $\mu_{n}(B_{n,i}) \geq \tfrac{1}{k_{n}+1}$ for each $i \in \{ 0,\ldots,k_{n}-1 \}$.
\end{enumerate} Then~(0) and~(1) hold with respect to the sequence $\mathcal{B}_{n} \defeq \{ B_{n,i} \mid i \in \{ 0,\ldots,k_{n}-1 \} \}$ ($n \in \mathbb{N}$). Moreover, \begin{align*}
	& \mu_{n}\!\left( \bigcup \mathcal{B}_{n} \right) \, = \, \sum\nolimits_{i=0}^{k_{n}-1} \mu_{n}(B_{n,i}) \, \geq \, \tfrac{k_{n}}{k_{n}+1} , & \sup \{ \mu_{n}(B) \mid B \in \mathcal{B}_{n} \} \, \leq \, 1 - \tfrac{k_{n}-1}{k_{n}+1}
\end{align*} for every $n \in \mathbb{N}$, which entails that $(\mathcal{B}_{n})_{n \in \mathbb{N}}$ also satisfies~(3) and~(4).

($\Longleftarrow$) Consider an integer $m \geq 1$ and real numbers $\kappa_{0},\ldots,\kappa_{m} > 0$ such that $\sum_{i=0}^{m} \kappa_{i} < 1$. To prove that $\liminf_{n \to \infty} \mathrm{Sep}(\mathcal{X}_{n};\kappa_{0},\ldots,\kappa_{n}) \geq \delta$, consider any $\tau \in (0,\delta)$. By our hypothesis, there exists a sequence $(\mathcal{B}_{n})_{n \in \mathbb{N}}$ satisfying~(0)--(3) with respect to $(\mathcal{X}_{n})_{n \in \mathbb{N}}$ and $\tau$. Since $\mathbb{Q}^{m+1}$ is dense in $\mathbb{R}^{m+1}$, we find positive integers $q,p_{0},\ldots,p_{m}$ such that $\sum_{i=0}^{m} p_{i} < q$ and $\kappa_{i} \leq p_{i}q^{-1}$ for each $i \in \{ 0,\ldots,m \}$. Upon multiplying $q,p_{0},\ldots,p_{m}$ by a suitable positive integer, we may furthermore assume that $m+2 + \sum_{i=0}^{m}p_{i} \leq q$. By~(2) and~(3), there is $\ell \in \mathbb{N}$ such that \begin{equation}\tag{$\ast$}\label{asymptotics}
	\forall n \in \mathbb{N}_{\geq \ell} \colon \qquad \mu_{n}\!\left(\bigcup \mathcal{B}_{n}\right) \, \geq \, 1 - \tfrac{1}{q} , \qquad \sup \{ \mu_{n}(B) \mid B \in \mathcal{B}_{n} \} \, \leq \, \tfrac{1}{q} .
\end{equation} We now claim that \begin{displaymath}
	\forall n \in \mathbb{N}_{\geq \ell} \colon \qquad \mathrm{Sep}(\mathcal{X}_{n};\kappa_{0},\ldots,\kappa_{m}) \, \geq \, \tau .
\end{displaymath} Let $n \in \mathbb{N}_{\geq \ell}$. Thanks to~\eqref{asymptotics}, there exists a subset $\mathcal{C}_{0} \subseteq \mathcal{B}_{n}$ such that $\tfrac{p_{0}}{q} \leq \mu_{n}(\bigcup \mathcal{C}_{0}) < \tfrac{p_{0}+1}{q}$. Recursively, for $\{ 1,\ldots,m \}$, having chosen pairwise disjoint subsets $\mathcal{C}_{0},\ldots,\mathcal{C}_{i-1} \subseteq \mathcal{B}_{n}$ so that \begin{displaymath}
	\tfrac{p_{0}}{q} \, \leq \, \mu_{n}(\mathcal{C}_{0}) \, < \, \tfrac{p_{0}+1}{q} , \quad \ldots , \quad \tfrac{p_{i-1}}{q} \, \leq \, \mu_{n}(\mathcal{C}_{i-1}) \, < \, \tfrac{p_{i-1}+1}{q} ,
\end{displaymath} we may apply~\eqref{asymptotics} to find a subset $\mathcal{C}_{i} \subseteq \mathcal{B}_{n} \setminus (\mathcal{C}_{0} \cup \ldots \cup \mathcal{C}_{i-1})$ such that $\tfrac{p_{i}}{q} \leq \mu_{n}(\bigcup \mathcal{C}_{i}) < \tfrac{p_{i}+1}{q}$. For each $i \in \{ 0,\ldots,m \}$, we obtain a Borel subset $D_{i} \defeq \bigcup \mathcal{C}_{i} \subseteq X_{n}$ with $\mu_{n}(D_{i}) \geq p_{i}q^{-1} \geq \kappa_{i}$. Moreover, $\inf \{ d_{n}(x,y) \mid x \in D_{i}, \, y \in D_{j} \} \geq \tau$ for any two distinct $i,j \in \{ 0,\ldots,m \}$, which entails that $\mathrm{Sep}(\mathcal{X}_{n};\kappa_{0},\ldots,\kappa_{m}) \geq \tau$, as desired. This shows that $(\mathcal{X}_{n})_{n \in \mathbb{N}}$ $\delta$-dissipates.\end{proof}

As mentioned above, dissipation is a strong form of non-concentration. Making this evident will require some preliminary considerations. We start off with a fairly general fact.

\begin{lem} Let $(X,d,\mu)$ be a fully supported $mm$-space. Then $(\mathrm{Lip}_{\ell}^{s}(X,d),\mathrm{me}_{\mu})$ is a compact metric space for any two real numbers $\ell,s \geq 0$. \end{lem}

\begin{proof} Note that $\mathrm{me}_{\mu}$ is a metric on $\mathrm{Lip}_{\ell}^{s}(X,d)$, because $\spt \mu = X$. Since $\mathrm{Lip}_{\ell}^{s}(X,d)$ is a compact subset of the product space $\mathbb{R}^{X}$ and equicontinuous, the Arzel\`{a}-Ascoli theorem, in the form of~\cite[7.15, pp.~232]{kelley}, asserts that $\mathrm{Lip}_{\ell}^{s}(X,d)$ is compact with respect to the topology $\tau_{C}$ of uniform convergence on compact subsets of $X$. To prove compactness of the metric space $(\mathrm{Lip}_{\ell}^{s}(X,d),\mathrm{me}_{\mu})$, it thus suffices to show that the topology $\tau_{M}$ generated by the metric $\mathrm{me}_{\mu}$ on $\mathrm{Lip}_{\ell}^{s}(X,d)$ is contained in $\tau_{C}$. To this end, let $U \in \tau_{M}$. Consider any $f \in U$. As $U \in \tau_{M}$, there exists $\epsilon > 0$ such that $B_{\mathrm{me}_{\mu}}(f,\epsilon) \subseteq U$. Being a Borel probability measure on a Polish space, $\mu$ must be tight (see, e.g.,~\cite[Chapter~II, Theorem~3.2]{part}). Hence, there exists a compact subset $K \subseteq X$ with $\mu (K) > 1-\epsilon$. In turn, \begin{displaymath}
	\{ g \in \mathrm{Lip}_{\ell}^{s}(X,d) \mid \sup\nolimits_{x \in K} \vert f(x) - g(x) \vert < \epsilon \} \, \subseteq \, B_{\mathrm{me}_{\mu}}(f,\epsilon) \, \subseteq \, U ,
\end{displaymath} which entails that $U$ is a neighborhood of $f$ in $\tau_{C}$. This shows that $U \in \tau_{C}$. Thus, $\tau_{M} \subseteq \tau_{C}$ as desired. (In fact, $\tau_{M} = \tau_{C}$ since $\tau_{M}$ is Hausdorff and $\tau_{C}$ is compact.) \end{proof}

Our next observation relates the observable distance of $mm$-spaces with the Gromov-Hausdorff distance of the corresponding spaces of bounded Lipschitz functions.

\begin{lem}\label{lemma:gromov.vs.hausdorff} Let $\mathcal{X}_{0} = (X_{0},d_{0},\mu_{0})$ and $\mathcal{X}_{1} = (X_{1},d_{1},\mu_{1})$ be two fully supported $mm$-spaces. For any two real numbers $\ell \geq 1$ and $s \geq 0$, \begin{displaymath}
	d_{\mathrm{GH}}\bigl(\bigl(\mathrm{Lip}_{\ell}^{s}(X_{0},d_{0}),\mathrm{me}_{\mu_{0}}\bigr),\bigl(\mathrm{Lip}_{\ell}^{s}(X_{1},d_{1}),\mathrm{me}_{\mu_{1}}\bigr)\bigr) \, \leq \, \ell d_{\mathrm{conc}}(\mathcal{X}_{0},\mathcal{X}_{1}) .
\end{displaymath} \end{lem}

\begin{proof} For each $i \in \{ 0,1 \}$, consider an arbitrary parametrization $\phi_{i}$ of $\mathcal{X}_{i}$. Let \begin{displaymath}
	Z \defeq \left(\mathrm{Lip}^{s}_{\ell}(X_{0},d_{0}) \circ \phi_{0}\right) \cup \left(\mathrm{Lip}^{s}_{\ell}(X_{1},d_{1}) \circ \phi_{1}\right) .
\end{displaymath} For each $i \in \{ 0,1 \}$, since $\mu_{i} = (\phi_{i})_{\ast}(\lambda)$, the map \begin{displaymath}
	\Phi_{i} \colon (\mathrm{Lip}^{s}_{\ell}(X_{i},d_{i}),\mathrm{me}_{\mu_{i}} ) \, \longrightarrow \, (Z,\mathrm{me}_{\lambda}) , \quad f \, \longmapsto \, f \circ \phi_{i}
\end{displaymath} is an isometric embedding, cf.~\cite[Lemma~5.31(1)]{ShioyaBook}. Furthermore, \begin{displaymath}
	\mathrm{H}_{\mathrm{me}_{\lambda}}(\mathrm{Lip}^{s}_{\ell}(X_{0},d_{0}) \circ \phi_{0},\mathrm{Lip}^{s}_{\ell}(X_{1},d_{1})\circ \phi_{1}) \leq \ell \mathrm{H}_{\mathrm{me}_{\lambda}}(\mathrm{Lip}_{1}(X_{0},d_{0}) \circ \phi_{0},\mathrm{Lip}_{1}(X_{1},d_{1})\circ \phi_{1}) .
\end{displaymath} Indeed, if $\mathrm{H}_{\mathrm{me}_{\lambda}}(\mathrm{Lip}_{1}(X_{0},d_{0}) \circ \phi_{0},\mathrm{Lip}_{1}(X_{1},d_{1})\circ \phi_{1}) < \delta$ for some number $\delta \in \mathbb{R}$, and we let $\{ i,j\} = \{ 0,1 \}$ and $f \in \mathrm{Lip}^{s}_{\ell}(X_{i},d_{i})$, then $\ell^{-1}f \in \mathrm{Lip}_{1}(X_{i},d_{i})$, thus there is $g \in \mathrm{Lip}_{1}(X_{j},d_{j})$ with $\mathrm{me}_{\lambda}((\ell^{-1}f) \circ \phi_{i},g \circ \phi_{j}) < \delta$, wherefore $h \defeq ((\ell g) \wedge s) \vee (-s) \in \mathrm{Lip}^{s}_{\ell}(X_{j},d_{j})$ and \begin{displaymath}
	\mathrm{me}_{\lambda}(f \circ \phi_{i},h \circ \phi_{j}) \, \leq \, \mathrm{me}_{\lambda}(f \circ \phi_{i},(\ell g) \circ \phi_{j}) \, \leq \, \ell \mathrm{me}_{\lambda}((\ell^{-1}f) \circ \phi_{i},g \circ \phi_{j}) \, < \, \ell \delta ,
\end{displaymath} which shows that $\mathrm{H}_{\mathrm{me}_{\lambda}}(\mathrm{Lip}^{s}_{\ell}(X_{0},d_{0}) \circ \phi_{0},\mathrm{Lip}^{s}_{\ell}(X_{1},d_{1})\circ \phi_{1}) \leq \ell \delta$. Consequently, \begin{align*}
	d_{\mathrm{GH}}\bigl(\bigl(\mathrm{Lip}_{\ell}^{s}(X_{0},d_{0}),\mathrm{me}_{\mu_{0}}\bigr),\bigl(\mathrm{Lip}_{\ell}^{s}&(X_{1},d_{1}),\mathrm{me}_{\mu_{1}}\bigr)\bigr) \\
	&\leq \, \mathrm{H}_{\mathrm{me}_{\lambda}}(\Phi_{0}(\mathrm{Lip}^{s}_{\ell}(X_{0},d_{0})),\Phi_{1}(\mathrm{Lip}^{s}_{\ell}(X_{1},d_{1}))) \\
	&= \,  \mathrm{H}_{\mathrm{me}_{\lambda}}(\mathrm{Lip}^{s}_{\ell}(X_{0},d_{0}) \circ \phi_{0},\mathrm{Lip}^{s}_{\ell}(X_{1},d_{1})\circ \phi_{1}) \\
	&\leq \, \ell \mathrm{H}_{\mathrm{me}_{\lambda}}(\mathrm{Lip}_{1}(X_{0},d_{0}) \circ \phi_{0},\mathrm{Lip}_{1}(X_{1},d_{1})\circ \phi_{1}) .
\end{align*} According to the definition of $d_{\mathrm{conc}}$, this completes the proof. \end{proof}

Our proof of Lemma~\ref{lemma:capacity} will involve Rademacher functions. Recall that, for any $n \in \mathbb{N}\setminus\{ 0 \}$, the \emph{$n$-th Rademacher function} is defined as \begin{displaymath}
	r_{n} \colon [0,1) \, \longrightarrow \, \{ -1, 1 \}, \quad t \, \longmapsto \, (-1)^{\lfloor 2^{n}t \rfloor} ,
\end{displaymath}  that is, $r_{n}(t) = (-1)^{k}$ whenever $t \in [2^{-n}k,2^{-n}(k+1))$ and $k \in \{ 0,\ldots,2^{n}-1 \}$. Let us note a well-known, elementary fact about this family of functions. Given any $n \in \mathbb{N}\setminus \{ 0 \}$, we will henceforth abbreviate $T_{n} \defeq \{ 2^{-n}k \mid k \in \{ 0,\ldots,2^{n}-1 \} \}$.

\begin{lem}\label{lemma:rademacher} Let $n \in \mathbb{N}\setminus \{ 0 \}$. For any two distinct $i,j \in \{ 1,\ldots,n \}$, \begin{displaymath}
	\vert \{ t \in T_{n} \mid r_{i}(t) = r_{j}(t) \} \vert \, = \, 2^{n-1} \, = \, \tfrac{\vert T_{n} \vert}{2} .
\end{displaymath} \end{lem}

\begin{proof} We include a proof for the sake of completeness. Let us briefly agree on some convenient notation: given a finite subset $T \subseteq [0,1)$, let $[x,y)_{T} \defeq \{ t \in T \mid x \leq t < y \}$ for $x,y \in [0,1]$, and define $\sigma_{T} \colon [0,1) \to [0,1], \, x \mapsto \min \{ t \in T \cup \{ 1 \} \mid x < t \}$. Without loss of generality, we may assume that $i < j$. For every $x \in T_{i}$, \begin{enumerate}
	\item[$(1)$] \ $\vert \{ t \in [x,\sigma_{T_{i}}(x))_{T_{n}} \! \mid r_{i}(t) = r_{j}(t) \} \vert \, = \, 2^{n-j}\vert \{ y \in [x,\sigma_{T_{i}}(x))_{T_{j}} \! \mid r_{i}(y) = r_{j}(y) \} \vert$, \smallskip
	\item[$(2)$] \ $\vert \{ y \in [x,\sigma_{T_{i}}(x))_{T_{j}} \! \mid r_{i}(y) = r_{j}(y) \} \vert \, = \, 2^{j-i-1}$.
\end{enumerate} In order to prove~(1) and~(2), let $x \in T_{i}$. Then~(2) follows by observing that $r_{i}$ is constant on the $2^{j-i}$-element set $[x,\sigma_{T_{i}}(x))_{T_{j}}$, whereas $r_{j}(\sigma_{T_{j}}(y)) = - r_{j}(y)$ for every $y \in [x,\sigma_{T_{i}}(x))_{T_{j}}$. With regards to~(1), we note that $[x,\sigma_{T_{i}}(x))_{T_{n}} = \bigcupdot \{ [y,\sigma_{T_{j}}(y))_{T_{n}} \! \mid y \in [x,\sigma_{T_{i}}(x))_{T_{j}} \}$ and that both $r_{i}$ and $r_{j}$ are constant on each of the $2^{n-j}$-element sets $[y,\sigma_{T_{j}}(y))_{T_{n}}$ ($y \in [x,\sigma_{T_{i}}(x))_{T_{j}}$), which entails that \begin{align*}
	\vert \{ t \in [x,\sigma_{T_{i}}(x))_{T_{n}} \! \mid r_{i}(t) = r_{j}(t) \} \vert \, & = \, \sum\nolimits_{y \in [x,\sigma_{T_{i}}(x))_{T_{j}} \!} \vert \{ t \in [y,\sigma_{T_{j}}(y))_{T_{n}} \! \mid r_{i}(t) = r_{j}(t) \} \vert \\
	&= \, 2^{n-j} \vert \{ y \in [x,\sigma_{T_{i}}(x))_{T_{j}} \! \mid r_{i}(y) = r_{j}(y) \} \vert .
\end{align*} Combining~(1) and~(2) with the fact that $T_{n} = \bigcupdot \{ [x,\sigma_{T_{i}}(x))_{T_{n}} \! \mid x \in T_{i} \}$, we arrive at \begin{align*}
	\vert \{ t \in T_{n} \mid r_{i}(t) = r_{j}(t) \} \vert \, & = \, \sum\nolimits_{x \in T_{i}} \vert \{ t \in [x,\sigma_{T_{i}}(x))_{T_{n}} \! \mid r_{i}(t) = r_{j}(t) \} \vert \\
	&\stackrel{(1)}{=} \, 2^{n-j}\sum\nolimits_{x \in T_{i}} \vert \{ y \in [x,\sigma_{T_{i}}(x))_{T_{j}} \! \mid r_{i}(y) = r_{j}(y) \} \vert \\
	&\stackrel{(2)}{=} \, 2^{n-j}2^{j-i-1}\vert T_{i} \vert \, = \, 2^{n-1} \, = \, \tfrac{\vert T_{n} \vert}{2} .\qedhere
\end{align*} \end{proof}

Now we are prepared for proving the key Lemma~\ref{lemma:capacity}.

\begin{lem}\label{lemma:capacity} Let $\mathcal{X} = (X,d,\mu)$ be a fully supported $mm$-space. For all $n \in \mathbb{N}\setminus \{ 0 \}$, $\delta \in (0,\infty)$, $\epsilon \in (0,1)$ and $\tau \in \bigl(0,\tfrac{1}{2}\bigr)$, \begin{displaymath}
	\mathrm{Sep}_{2^{n}-1}(\mathcal{X};(1-\epsilon)2^{-n}) \, \geq \, \delta \quad \Longrightarrow \quad \mathrm{Cap}_{(1 - \epsilon)\tau}\!\left(\mathrm{Lip}_{\delta^{-1}}^{1}(X,d),\mathrm{me}_{\mu}\right) \, \geq \, n .
\end{displaymath} \end{lem}

\begin{proof} Assume that $\mathrm{Sep}_{2^{n}-1}(\mathcal{X};(1-\epsilon)2^{-n}) \geq \delta$. Then there exist (necessarily disjoint) Borel subsets $B_{0},\ldots,B_{2^{n}-1} \subseteq X$ such that \begin{enumerate}
	\item[$(1)$] \, $\mu (B_{i}) \geq (1-\epsilon)2^{-n}$ for each $i \in \{ 0,\ldots,2^{n}-1\}$, and \smallskip
	\item[$(2)$] \, $\inf \{ d(x,y) \mid x \in B_{i}, \, y \in B_{j} \} \geq \delta$ for any two distinct $i,j \in \{ 0,\ldots,2^{n} -1 \}$.
\end{enumerate} Consider the Borel set $B \defeq \bigcup_{i=0}^{2^{n}-1} B_{i} \subseteq X$, and note that $\mu (B) \geq 1-\epsilon$, as follows from~(1) and the fact that $B_{0},\ldots,B_{2^{n}-1}$ are pairwise disjoint. Let $\pi \colon B \to \{ 0,\ldots,2^{n}-1\}$ be the unique map with $\pi^{-1}(i) = B_{i}$ for all $i \in \{ 0,\ldots,2^{n}-1 \}$. For each $i \in \{ 1,\ldots,n \}$, consider the $i$-th Rademacher function $r_{i} \colon [0,1) \to \{ -1,1 \}$ and let \begin{displaymath}
	f_{i} \colon B \, \longrightarrow \, \{ 0,1\} , \quad x \, \longmapsto \, \tfrac{1}{2}(1 + r_{i}(2^{-n}\pi(x))) .
\end{displaymath} As each of the functions $f_{1},\ldots,f_{n} \colon B \to \{ 0,1 \}$ is constant on each of the sets $B_{0},\ldots,B_{2^{n}-1}$, assertion~(2) implies that $\{ f_{1},\ldots,f_{n} \} \subseteq \mathrm{Lip}_{\delta^{-1}}^{1}(B,d\!\!\upharpoonright_{B})$. Utilizing a standard construction, for each $i \in \{ 1,\ldots,n \}$ we define \begin{displaymath}
	f^{\ast}_{i} \colon X \, \longrightarrow \, [0,1], \quad x \, \longmapsto \, \left(\inf\nolimits_{y \in B} f_{i}(y) + \delta^{-1}d(x,y)\right)\wedge 1 
\end{displaymath} and observe that $f^{\ast}_{i} \in \mathrm{Lip}_{\delta^{-1}}^{1}(X,d)$ and $f^{\ast}_{i}\vert_{B} = f_{i}$. Define $T_{n} \defeq \{ 2^{-n}k \mid k \in \{ 0,\ldots,2^{n}-1\} \}$ as in Lemma~\ref{lemma:rademacher}. For any two distinct $i,j \in \{ 1,\ldots,n \}$, we consider the Borel set \begin{align*}
	N_{ij} \, &\defeq \, \{ x \in B \mid \vert f_{i}(x) - f_{j}(x)\vert = 1 \} \, = \, \{ x \in B \mid f_{i}(x) \ne f_{j}(x) \} \\
	&\qquad \qquad \qquad \qquad \qquad \qquad \qquad \qquad = \, \bigcup \{ B_{2^{n}t} \mid t \in T_{n} , \, r_{i}(t) \ne r_{j}(t) \}
\end{align*} and conclude that $\mu (N_{ij}) \geq 2^{n-1}(1-\epsilon)2^{-n} = (1-\epsilon)2^{-1}$, taking into account assertion~(1), the pairwise disjointness of $B_{0},\ldots,B_{2^{n}-1}$, and Lemma~\ref{lemma:rademacher}. It follows that \begin{displaymath}
	\mathrm{me}_{\mu}(f^{\ast}_{i},f^{\ast}_{j}) \, \geq \, \mu (N_{ij}) \, \geq \, 2^{-1}(1-\epsilon) \, > \, (1-\epsilon)\tau
\end{displaymath} for any two distinct $i,j \in \{ 1,\ldots,n \}$, i.e., $\{ f^{\ast}_{1},\ldots,f^{\ast}_{n} \}$ is $(1-\epsilon)\tau$-discrete in $\!\left(\mathrm{Lip}_{\delta^{-1}}^{1}(X,d),\mathrm{me}_{\mu}\right)\!$. Thus, $\mathrm{Cap}_{(1 - \epsilon)\tau}\!\left(\mathrm{Lip}_{\delta^{-1}}^{1}(X,d),\mathrm{me}_{\mu}\right) \geq n$ as desired. \end{proof}

Everything is prepared to show that dissipation does indeed constitute a strong opposite to concentration.

\begin{prop}\label{proposition:dissipation.vs.concentration} Let $\mathcal{X}_{n} = (X_{n},d_{n},\mu_{n})$ $(n \in \mathbb{N})$ be a sequence of fully supported $mm$-spaces $\delta$-dissipating for some $\delta > 0$. Then, for all $\ell \in \bigl(\delta^{-1}, \infty\bigr)$ and $\alpha \in \bigl( 0,\tfrac{1}{2} \bigr)$, \begin{displaymath}
	\mathrm{Cap}_{\alpha}\!\left(\mathrm{Lip}^{1}_{\ell}(X_{n},d_{n}),\mathrm{me}_{\mu_{n}}\right) \, \longrightarrow \, \infty \quad (n \longrightarrow \infty) .
\end{displaymath} \end{prop}

\begin{proof} Let $\ell \in \bigl(\delta^{-1}, \infty\bigr)$, $\alpha \in \bigl( 0,\tfrac{1}{2} \bigr)$, $m \in \mathbb{N}\setminus \{ 0 \}$. Choose $\epsilon \in (0,1)$ and $\tau \in \bigl( 0,\tfrac{1}{2} \bigr)$ such that $(1-\epsilon)\tau \geq \alpha$. Since $(\mathcal{X}_{n})_{n \in \mathbb{N}}$ $\delta$-dissipates, we find $n_{0} \in \mathbb{N}$ with $\mathrm{Sep}_{2^{m}-1}(\mathcal{X}_{n};(1-\epsilon)2^{-m}) \geq \ell^{-1}$ for all $n \in \mathbb{N}$, $n \geq n_{0}$. By Lemma~\ref{lemma:capacity}, it follows that \begin{displaymath}
	\mathrm{Cap}_{\alpha}\!\left(\mathrm{Lip}_{\ell}^{1}(X_{n},d_{n}),\mathrm{me}_{\mu_{n}}\right) \, \geq \, \mathrm{Cap}_{(1 - \epsilon)\tau}\!\left(\mathrm{Lip}_{\ell}^{1}(X_{n},d_{n}),\mathrm{me}_{\mu_{n}}\right) \, \geq \, m
\end{displaymath} for every $n \in \mathbb{N}$ with $n \geq n_{0}$, which proves our claim. \end{proof}
	
Combining Proposition~\ref{proposition:dissipation.vs.concentration} with Theorem~\ref{theorem:gromov.precompactness}, we arrive at the following.
	
\begin{cor}\label{corollary:dissipation.vs.concentration} If a sequence of $mm$-spaces dissipates, then it does not have a $d_{\mathrm{conc}}$-Cauchy subsequence. \end{cor}

\begin{proof} Since dissipation is inherited by subsequences, it suffices to check that no dissipating sequence of $mm$-spaces can possibly be $d_{\mathrm{conc}}$-Cauchy. Moreover, since both dissipation and being Cauchy with respect to $d_{\mathrm{conc}}$ are invariant under $mm$-space isomorphisms and every $mm$-space is isomorphic to a fully supported one, it is sufficient to consider a sequence of fully supported $mm$-spaces $\mathcal{X}_{n} = (X_{n},d_{n},\mu_{n})$ $(n \in \mathbb{N})$. If $(\mathcal{X}_{n})_{n \in \mathbb{N}}$ $\delta$-dissipates for some $\delta > 0$, then Proposition~\ref{proposition:dissipation.vs.concentration} along with Theorem~\ref{theorem:gromov.precompactness} asserts that $\bigl(\mathrm{Lip}_{1+\delta^{-1}}^{1}(X_{n},d_{n}),\mathrm{me}_{\mu_{n}}\bigr)_{n \in \mathbb{N}}$ is not $d_{\mathrm{GH}}$-Cauchy, which, according to Lemma~\ref{lemma:gromov.vs.hausdorff}, implies that $(\mathcal{X}_{n})_{n \in \mathbb{N}}$ is not $d_{\mathrm{conc}}$-Cauchy. \end{proof}

\begin{remark}\label{remark:dissipation.vs.concentration} The converse of Corollary~\ref{corollary:dissipation.vs.concentration} does not hold. In fact, letting $\lambda_{n} \defeq \lambda^{\otimes n}$ and \begin{displaymath}
	d_{n} \colon [0,1]^{n} \times [0,1]^{n} \, \longrightarrow \, \mathbb{R}, \quad (x,y) \, \longmapsto \, \sup \{ \vert x_{i} - y_{i} \vert \mid i \in \{ 1,\ldots,n \} \}
\end{displaymath} for each $n \in \mathbb{N}$, the sequence $([0,1]^{n},d_{n},\lambda_{n})_{n \in \mathbb{N}}$ does not dissipate~\cite[Theorem~8.8]{ShioyaBook}, but does not contain a $d_{\mathrm{conc}}$-Cauchy subsequence either, as follows by the argument in the proof of~\cite[Proposition~7.36]{ShioyaBook}. \end{remark}

Since the present note is aimed at exhibiting instances of dissipation in topological groups, we will subsequently add some remarks about dissipation in the very general framework of uniform spaces, thus following the lead of Pestov, who in~\cite[Definition~2.6]{pestov02} adapted the L\'evy concentration property from the realm of $mm$-spaces to Borel probability measures in uniform spaces. We provide a corresponding generalization of Gromov's concept of dissipation. Throughout the present note, all uniform spaces are assumed to be \emph{Hausdorff}.

\begin{definition} Let $X$ be a uniform space. For an entourage $U$ in $X$, we will say that a net $(\mu_{i})_{i \in I}$ of Borel probability measures on $X$ \emph{$U$-dissipates in $X$} if there exists a family $(\mathcal{B}_{i})_{i \in I}$ of finite sets of Borel subsets of $X$ such that \begin{enumerate}
	\item[$(1)$] \, $(B \times C) \cap U = \emptyset$ for every $i \in I$ and any two distinct $B,C \in \mathcal{B}_{i}$,
	\item[$(2)$] \, $\lim_{i \to I}\mu_{i}(\bigcup \mathcal{B}_{i}) = 1$, and
	\item[$(3)$] \, $\lim_{i \to I} \sup \{ \mu_{i}(B) \mid B \in \mathcal{B}_{i} \} = 0$.
\end{enumerate} A net of Borel probability measures on $X$ will be said to \emph{dissipate in $X$} if it $U$-dissipates for some entourage $U$ of $X$. \end{definition}

Of course, dissipation and L\'evy concentration~\cite[Definition~2.6]{pestov02} are two mutually exclusive properties for a net of Borel probability measures on a uniform space to have. Let us point out the connection with dissipation of $mm$-spaces.

\begin{prop}\label{proposition:uniform.dissipation} Let $(X,d)$ be any metric space and let $\mathcal{X}_{n} = (X_{n},d_{n},\mu_{n})$ $(n \in \mathbb{N})$ be a sequence of $mm$-spaces. For each $n \in \mathbb{N}$, let $\phi_{n} \colon (X_{n},d_{n}) \to (X,d)$ be an isometric embedding. Then $(\mathcal{X}_{n})_{n \in \mathbb{N}}$ dissipates if and only if the sequence $((\phi_{n})_{\ast}(\mu_{n}))_{n \in \mathbb{N}}$ dissipates in $(X,d)$. \end{prop}

\begin{proof} This is an immediate consequence of Lemma~\ref{lemma:dissipation.revisited}. \end{proof}

We conclude this section with a clarifying remark about the existence of dissipating families of Borel probability measures in uniform spaces. To agree on some terminology and notation, let $X$ be a uniform space. For an entourage $U$ of $X$, we let \begin{displaymath}
	U[A] \defeq \{ y \in X \mid \exists x \in A \colon \, (x,y) \in U \} \qquad (A \subseteq X) .
\end{displaymath} Recall that $X$ is \emph{precompact} if, for every entourage $U$ of $X$, there exists a finite subset $F \subseteq X$ such that $X = U[F]$. Moreover, for an entourage $U$ of $X$, a subset $B \subseteq X$ is called \emph{$U$-discrete} if $(B \times B) \cap U = \Delta_{B}$. It is easy to see that $X$ is precompact if and only if the quantity \begin{displaymath}
	\mathrm{Cap}_{U}(X) \defeq \sup \{ \vert B \vert \mid B \subseteq X \text{ $U$-discrete} \}
\end{displaymath} is finite for every entourage $U$ of $X$.

\begin{prop}\label{proposition:dissipation.in.uniform.spaces} A uniform space $X$ is non-precompact if and only if $X$ admits a dissipating net of (finitely supported regular) Borel probability measures. \end{prop}

\begin{proof} ($\Longrightarrow$) Being non-precompact, $X$ admits an entourage $U$ with $\mathrm{Cap}_{U}(X) = \infty$, whence we find a sequence $(F_{n})_{n \in \mathbb{N}}$ of finite non-empty $U$-discrete subsets of $X$ such that $\vert F_{n} \vert \longrightarrow \infty$ as $n \longrightarrow \infty$. The sequence of finitely supported regular Borel probability measures $(\delta_{F_{n}})_{n \in \mathbb{N}}$ thus $U$-dissipates in $X$, as witnessed by the sequence $\mathcal{B}_{n} \defeq \{ \{ x \} \mid x \in F_{n} \}$ $(n \in \mathbb{N})$.
	
($\Longleftarrow$) Suppose that $X$ admits a net $(\mu_{i})_{i \in I}$ of Borel probability measures $U$-dissipating~in~$X$ for some entourage $U$ of $X$, i.e., there is a family $(\mathcal{B}_{i})_{i \in I}$ of finite sets of Borel subsets of~$X$ such that \begin{enumerate}
	\item[$(1)$] \, $(B \times C) \cap U = \emptyset$ for every $i \in I$ and any two distinct $B,C \in \mathcal{B}_{i}$,
	\item[$(2)$] \, $\lim_{i \to I}\mu_{i}(\bigcup \mathcal{B}_{i}) = 1$, and
	\item[$(3)$] \, $\lim_{i \to I} \sup \{ \mu_{i}(B) \mid B \in \mathcal{B}_{i} \} = 0$.
\end{enumerate} We show that $\mathrm{Cap}_{U}(X) = \infty$. To this end, let $m \in \mathbb{N}\setminus \{ 0 \}$. There exists $i \in I$ such that \begin{equation}\tag{$\ast$}\label{uniform}
	\mu_{i}\!\left( \bigcup \mathcal{B}_{i} \right) \, > \, 1-\tfrac{1}{m} , \qquad \quad \sup \{ \mu_{i}(B) \mid B \in \mathcal{B}_{i} \} \, \leq \, \tfrac{1}{m} .
\end{equation} Select a subset $F \subseteq \bigcup \mathcal{B}_{i}$ such that $\vert F \cap B \vert = 1$ for each $B \in \mathcal{B}_{i}$. Clearly, $F$ is $U$-discrete by~(1). Moreover, $\vert F \vert \geq m$ due to~\eqref{uniform}, and thus $\mathrm{Cap}_{U}(X) \geq m$. Hence, $X$ is not precompact. \end{proof}

\section{Topological groups: convergence to invariance}\label{section:invariance}

The purpose of this section is to recompile some results of~\cite{SchneiderThom} concerning amenability of topological groups and UEB-convergence to invariance. Throughout the present note, all topological groups are assumed to be \emph{Hausdorff}.

Let $X$ be a uniform space and consider the commutative unital real Banach algebra~$\mathrm{UCB}(X)$ of all bounded uniformly continuous real-valued functions on $X$ carrying the supremum~norm. The collection $\mathrm{M}(X)$ of all \emph{means} on $\mathrm{UCB}(X)$, that is, (necessarily continuous) positive unital linear forms on $\mathrm{UCB}(X)$, endowed with the weak-$\ast$ topology, i.e., the initial topology generated by all maps $\mathrm{M}(X) \to \mathbb{R}, \, \mu \mapsto \mu (f)$ where $f \in \mathrm{UCB}(X)$, is a compact Hausdorff space. The set $\mathrm{S}(X)$ of all (necessarily positive and linear) unital ring homomorphisms from $\mathrm{UCB}(X)$ into~$\mathbb{R}$ is a closed subspace thereof, commonly referred to as the \emph{Samuel compactification} of~$X$. A subset $H \subseteq \mathrm{UCB}(X)$ is called \emph{UEB} (short for \emph{uniformly equicontinuous bounded}) if $H$ is bounded in the supremum norm and \emph{uniformly equicontinuous}, i.e., for every $\epsilon > 0$ there exists an entourage $U$ of $X$ such that \begin{displaymath}
	\forall f \in H \ \forall (x,y) \in U \colon \qquad \vert f(x) - f(y) \vert \, \leq \, \epsilon .
\end{displaymath} The set $\mathrm{UEB}(X)$ of all UEB subsets of $\mathrm{UCB}(X)$ is a convex vector bornology on the vector space $\mathrm{UCB}(X)$. The \emph{UEB topology} on the continuous dual~$\mathrm{UCB}(X)^{\ast}$ is defined as the topology of uniform convergence on UEB subsets of $\mathrm{UCB}(X)$, which is a locally convex linear topology on the vector space $\mathrm{UCB}(X)^{\ast}$ containing the weak-$\ast$ topology, i.e., the initial topology generated by the maps $\mathrm{UCB}(X)^{\ast} \to \mathbb{R}, \, \mu  \mapsto \mu (f)$ where $f \in \mathrm{UCB}(X)$. A comprehensive account on the UEB topology is given in Pachl's book~\cite{PachlBook}.

Let $G$ be a topological group and denote by $\mathcal{U}(G)$ the neighborhood filter of the neutral element in $G$. We endow $G$ with its \emph{right uniformity} defined by the basic entourages \begin{displaymath}
	U_{\Rsh} \defeq \left\{ (x,y) \in G \times G \left\vert \, yx^{-1} \in U \right\} \qquad (U \in \mathcal{U}(G)) . \right.
\end{displaymath} In reference to this uniformity, we denote by $\mathrm{RUCB}(G)$ the set of all bounded uniformly continuous real-valued function on $G$ and by $\mathrm{RUEB}(G)$ the set of all UEB subsets of $\mathrm{RUCB}(G)$. Letting $\lambda_{g} \colon G \to G, \, x \mapsto gx$ for any $g \in G$, we note that $G$ acts continuously on $\mathrm{M}(G)$ by \begin{displaymath}
(g\mu)(f) \defeq \mu (f \circ \lambda_{g}) \qquad (g \in G, \, \mu \in \mathrm{M}(G), \, f \in \mathrm{RUCB}(G)) ,
\end{displaymath} and that $\mathrm{S}(G)$ is a $G$-invariant subspace of $\mathrm{M}(G)$. Recall that $G$ is \emph{amenable} (resp., \emph{extremely amenable}) if $\mathrm{M}(G)$ (resp., $\mathrm{S}(G)$) contains a $G$-fixed point. It is well known that $G$ is amenable (resp., extremely amenable) if and only if every continuous action of $G$ on a non-void compact Hausdorff space admits an invariant regular Borel probability measure (resp., a fixed point). For more on (extreme) amenability of topological groups, we refer to~\cite{PestovBook}.

We will need a characterization of amenability in terms of asymptotically invariant finitely supported probability measures from~\cite{SchneiderThom}.

\begin{definition} Let $G$ be a topological group. A net $(\mu_{i})_{i \in I}$ of Borel probability measures on $G$ is said to \emph{UEB-converge to invariance (over $G$)} if \begin{displaymath}
	\forall g \in G \ \forall H \in \mathrm{RUEB}(G) \colon \qquad \sup\nolimits_{f \in H} \left\lvert \int f \circ \lambda_{g} \, d\mu_{i} - \int f \, d\mu_{i} \right\rvert \longrightarrow 0 \quad (i \longrightarrow I) .
\end{displaymath} \end{definition}
	
\begin{thm}[\cite{SchneiderThom}, Theorem~3.2]\label{theorem:topological.day} A topological group is amenable if and only if it admits a net of (finitely supported, regular) Borel probability measures UEB-converging to invariance. \end{thm}

A very useful fact about UEB-convergence to invariance is the stability with respect to convolution of measures from the right, see Lemma~\ref{lemma:convoluted.convergence} below. For the proof, we will use the following Lemma~\ref{lemma:convoluted.functions}. For a Borel probability measure $\mu$ on a topological group $G$ and any bounded Borel function $f \colon G \to \mathbb{R}$, let us define ${}_{\mu}f \colon G \to \mathbb{R}, \, x \mapsto \int f(xy) \, d\mu(y)$.

\begin{lem}[cf.~\cite{PachlBook}, Lemma~9.1]\label{lemma:convoluted.functions} Let $G$ be a topological group. For every $H \in \mathrm{RUEB}(G)$, \begin{displaymath}
	\{ {}_{\mu}f \mid f \in H, \, \mu \in \mathrm{P}(G) \} \in \mathrm{RUEB}(G) .
\end{displaymath} \end{lem}

\begin{proof} Let $H \in \mathrm{RUEB}(G)$. Evidently, as $H$ is norm-bounded, $K \defeq \{ {}_{\mu} f \mid f \in H, \, \mu \in \mathrm{P}(G) \}$ must be, too. To prove uniform equicontinuity, let $\epsilon > 0$. Since $H$ is a member of $\mathrm{RUEB}(G)$, there exists $U \in \mathcal{U}(G)$ such that \begin{displaymath}
	\forall f \in H \, \forall x,y \in G \colon \qquad xy^{-1} \in U \ \Longrightarrow \ \vert f(x) - f(y) \vert \, \leq \, \epsilon .
\end{displaymath} Therefore, if $f \in H$ and $x,y \in G$ with $xy^{-1} \in U$, then $(xz)(yz)^{-1} = xy^{-1} \in U$ and hence $\vert f(xz) - f(yz) \vert \leq \epsilon$ for all $z \in G$, which implies that \begin{displaymath}
	\left\lvert \left( {}_{\mu}f \right) (x) - \left( {}_{\mu}f \right) (y) \right\rvert \, = \, \left\lvert \int f(xz) - f(yz) \, d\mu(z) \right\rvert \, \leq \, \int \lvert f(xz) - f(yz) \rvert \, d\mu(z) \, \leq \, \epsilon
\end{displaymath} for any $\mu \in \mathrm{P}(G)$. This shows that $K$ indeed belongs to $\mathrm{RUEB}(G)$. \end{proof}

Given a topological group $G$, we define the \emph{convolution} of two measures $\mu , \nu \in \mathrm{P}(G)$ to be the Borel probability measure $\mu \ast \nu \defeq m_{\ast}(\mu \otimes \nu)$, where $m \colon G \times G \to G, \, (x,y) \mapsto xy$.

\begin{lem}\label{lemma:convoluted.convergence} Let $(\mu_{i})_{i \in I}$ and $(\nu_{i})_{i \in I}$ be two nets of Borel probability measures on a topological group $G$. If $(\mu_{i})_{i \in I}$ UEB-converges to invariance over $G$, then so does $(\mu_{i} \ast \nu_{i})_{i \in I}$. \end{lem}

\begin{proof} Let $H \in \mathrm{RUEB}(G)$. Due to Lemma~\ref{lemma:convoluted.functions}, $\{ {}_{\nu}f \mid f \in H, \, \nu \in \mathrm{P}(G) \}$ belongs to $\mathrm{RUEB}(G)$. Hence, as $(\mu_{i})_{i \in I}$ UEB-converges to invariance over $G$, \begin{displaymath}
	\sup\nolimits_{\nu \in \mathrm{P}(G)}\sup\nolimits_{f \in H} \left\lvert \int \left( {}_{\nu} f\right) \circ \lambda_{g} \, d\mu_{i} - \int {}_{\nu} f \, d\mu_{i} \right\rvert \, \longrightarrow \, 0 \quad (i \longrightarrow I) .
\end{displaymath} for every $g \in G$. Since Fubini's theorem yields that \begin{displaymath}
	\int \left( {}_{\nu_{i}}f \right) \circ \lambda_{g} \, d\mu_{i} \, = \, \int \int f(gxy) \, d\nu_{i}(y) \, d\mu_{i}(x) \, = \, \int \int f(gxy) \, d\mu_{i}(x) \, d\nu_{i}(y) \, = \, \int f \circ \lambda_{g} \, d(\mu_{i} \ast \nu_{i})
\end{displaymath} for all $i \in I$, $f \in H$ and $g \in G$, it follows that \begin{align*}
	\sup\nolimits_{f \in H} \left\lvert \int f \circ \lambda_{g} \, d(\mu_{i} \, \ast \! \right. & \left. \nu_{i})- \int f \, d(\mu_{i} \ast \nu_{i}) \right\rvert \, = \, \sup\nolimits_{f \in H} \left\lvert \int \left( {}_{\nu_{i}}f \right) \circ \lambda_{g} \, d\mu_{i} -\int {}_{\nu_{i}}f \, d\mu_{i} \right\rvert \\
	&\leq \, \sup\nolimits_{\nu \in \mathrm{P}(G)}\sup\nolimits_{f \in H} \left\lvert \int \left( {}_{\nu} f\right) \circ \lambda_{g} \, d\mu_{i} - \int {}_{\nu} f \, d\mu_{i} \right\rvert \, \longrightarrow \, 0 \quad (i \to I) .
\end{align*} for every $g \in G$. This shows that $(\mu_{i} \ast \nu_{i})_{i \in I}$ UEB-converges to invariance over $G$. \end{proof}

\section{Equivariant dissipation}\label{section:groups}

The main result of this note reads as follows.

\begin{thm}\label{theorem:equivariant.dissipation} If a topological group $G$ admits an open subgroup of infinite index, then every net of tight Borel probability measures on $G$ UEB-converging to invariance dissipates in $G$. \end{thm}

\begin{proof} Let $H$ be any open subgroup of $G$. For any tight Borel probability measure $\mu$ on $G$, \begin{equation}\tag{$\ast$}\label{tight}
	\sup \{ \mu(HF) \mid F \subseteq G \text{ finite} \} \, = \, 1 .
\end{equation} To see this, let $\epsilon > 0$. As $\mu$ is tight, there exists a compact subset $K \subseteq G$ with $\mu(K) \geq 1-\epsilon$. Since $K$ is compact and $H$ is open, we find a finite subset $F \subseteq G$ such that $K \subseteq HF$, which readily implies that $\mu(HF) \geq \mu(K) \geq 1-\epsilon$, as desired.

Now suppose that $H$ has infinite index in $G$. We claim that, if $(\mu_{i})_{i \in I}$ is any net of Borel probability measures on $G$ UEB-converging to invariance, then \begin{equation}\tag{$\ast \ast$}\label{dissipation}
	\lim\nolimits_{i \to I} \sup \{ \mu_{i}(Hg) \mid g \in G \} \, = \, 0 .
\end{equation} For a subset $A \subseteq G$, we denote by $\chi_{A} \colon G \to \{ 0,1\}$ the corresponding characteristic function. Since $H$ is an open subgroup of $G$, we have $\{ \chi_{Hg} \mid g \in G \} \in \mathrm{RUEB}(G)$. Indeed, $\{ \chi_{Hg} \mid g \in G \}$ is clearly norm-bounded, and if $x,y \in G$ with $xy^{-1} \in H$, then $\chi_{Hg}(x) = \chi_{Hg}(y)$ for all $g \in G$. Thus, as $(\mu_{i})_{i \in I}$ UEB-converges to invariance over $G$, \begin{displaymath}
	\sup\nolimits_{g \in G} \vert \mu_{i}(fHg) - \mu_{i}(Hg) \vert \, = \, \sup\nolimits_{g \in G} \left\lvert \int \chi_{Hg} \circ \lambda_{f^{-1}} \, d\mu_{i} - \int \chi_{Hg} \, d\mu_{i} \right\rvert \, \longrightarrow \, 0 \quad (i \longrightarrow I) 
\end{displaymath} for all $f \in G$. In order to prove~\eqref{dissipation}, let $\epsilon > 0$. Pick any $n \in \mathbb{N}\setminus \{ 0 \}$ with $\tfrac{2}{n} \leq \epsilon$. Due to $H$ having infinite index in $G$, we find $F \subseteq G$ such that $\vert F \vert = n$ and $fH \cap f'H = \emptyset$ for any two distinct $f,f' \in F$. By the above, there exists $i_{0} \in I$ such that, for all $i \in I$  with $i \geq i_{0}$, \begin{displaymath}
	\sup\nolimits_{f \in F} \sup\nolimits_{g \in G} \vert \mu_{i}(fHg) - \mu_{i}(Hg) \vert \, \leq \, \tfrac{1}{n} .
\end{displaymath} For all $g \in G$ and $i \in I$ with $i \geq i_{0}$, it follows that \begin{displaymath}
	1 \, = \, \mu_{i}(G) \, \geq \, \sum\nolimits_{f \in F} \mu_{i}(fHg) \, \geq \, n \cdot \left( \mu_{i}(Hg) - \tfrac{1}{n} \right) \, = \, n\mu_{i}(Hg) - 1 ,
\end{displaymath} whence $\mu_{i}(Hg) \leq \tfrac{2}{n} \leq \epsilon$. That is, $\sup \{ \mu_{i}(Hg) \mid g \in G\} \leq \epsilon$ for all $i \in I$ with $i \geq i_{0}$, which proves the desired assertion~\eqref{dissipation}.

To conclude, let $(\mu)_{i \in I}$ be any net of tight Borel probability measures on $G$ UEB-converging to invariance over $G$. Thanks to~\eqref{tight}, we may choose finite $H_{\Rsh}$-discrete subsets $F_{i} \subseteq G$ ($i \in I$) such that $\lim_{i \to I} \mu_{i}(HF_{i}) = 1$. Moreover, $\lim_{i \to I} \sup \{ \mu_{i}(Hg) \mid g \in F_{i} \} = 0$ due to~\eqref{dissipation}. The family $\mathcal{B}_{i} \defeq \{ Hg \mid g \in F_{i} \}$~($i \in I$) hence witnesses that $(\mu_{i})_{i \in I}$ dissipates in $G$. \end{proof}

\begin{remarks} Let $G$ be a topological group.
	
(1) Concerning the hypothesis of Theorem~\ref{theorem:equivariant.dissipation}, note that $G$ admitting an open subgroup of infinite index is equivalent to the existence of a surjective continuous homomorphism from~$G$ onto some non-precompact, non-archimedean topological~group.
	
(2) Provided that $G$ is separable, the tightness condition for the measures in Theorem~\ref{theorem:equivariant.dissipation} may be removed, for if $H$ is any open subgroup of $G$, then $\{ Hg \mid g \in G \}$ is countable and thus \eqref{tight} in the proof of Theorem~\ref{theorem:equivariant.dissipation} holds for any Borel probability measure on $G$, by $\sigma$-additivity. This works more generally in case $G$ is \emph{$\tau$-narrow}~\cite[Section~5.1, p.~286]{AT} for some cardinal $\tau$ not \emph{real-valued measurable}~\cite[Appendix~3, Definition~1]{BB00}: if $H$ is an open subgroup of $G$, then the cardinality of $\{ Hg \mid g \in G \}$ is not real-valued measurable either, and consequently \cite[Chapter~8, Theorem~1]{BB00} asserts that, for any Borel probability measure $\mu$ on $G$, there exists a countable subset $C \subseteq G$ with $\mu (HC) = 1$, which implies~\eqref{tight} by $\sigma$-additivity of $\mu$. \end{remarks}

We proceed to consequences of Theorem~\ref{theorem:equivariant.dissipation}.

\begin{cor}\label{corollary:equivariant.dissipation} Let $G$ be a metrizable topological group along with a right-invariant compatible metric $d$ and let $(K_{n})_{n \in \mathbb{N}}$ be an ascending chain of compact subgroups such that $G = \overline{\bigcup_{n \in \mathbb{N}} K_{n}}$. For each $n \in \mathbb{N}$, we define $d_{n} \defeq d\!\!\upharpoonright_{K_{n}}$ and denote by $\mu_{n}$ the normalized Haar measure on~$K_{n}$. If $G$ admits an open subgroup of infinite index, then $(K_{n},d_{n},\mu_{n})_{n \in \mathbb{N}}$ dissipates, thus fails to have a $d_{\mathrm{conc}}$-Cauchy subsequence. \end{cor}

\begin{proof} Due to Theorem~\ref{theorem:equivariant.dissipation}, the measures $(\mu_{n})_{n \in \mathbb{N}}$, pushed forward to $G$ along the respective inclusion maps, dissipate in $G$. Being a compatible right-invariant metric, $d$ generates the right uniformity of $G$. Consequently, $(K_{n},d_{n},\mu_{n})_{n \in \mathbb{N}}$ must dissipate by Proposition~\ref{proposition:uniform.dissipation}. By Corollary~\ref{corollary:dissipation.vs.concentration}, this entails that $(K_{n},d_{n},\mu_{n})_{n \in \mathbb{N}}$ cannot have a $d_{\mathrm{conc}}$-Cauchy subsequence. \end{proof}

\begin{remark}\label{remark:equivariant.dissipation} Let $G$ be a metrizable topological group together with a left-invariant compatible metric $d$. Then \begin{displaymath}
	d^{-} \colon G \times G \, \longrightarrow \, \mathbb{R}, \quad (x,y) \, \longmapsto \, d\!\left(x^{-1},y^{-1}\right)
\end{displaymath} is a right-invariant compatible metric on $G$. Moreover, if $K$ is any compact subgroup of $G$ and $\mu$ denotes its normalized Haar measure, then $(K,d^{-}\!\!\upharpoonright_{K},\mu) \to (K,d\!\!\upharpoonright_{K},\mu), \, x \mapsto x^{-1}$ constitutes an isomorphism of $mm$-spaces. It follows that, in Corollary~\ref{corollary:equivariant.dissipation}, the word \emph{right-invariant} may equivalently be replaced by \emph{left-invariant}. \end{remark}

In view of Remark~\ref{remark:equivariant.dissipation}, our Corollary~\ref{corollary:equivariant.dissipation} resolves Problem~\ref{problem} in the negative.

\begin{cor}\label{corollary:final} For each $n \in \mathbb{N}$, let $\mu_{n}$ denote the normalized counting measure on $\mathrm{Sym}(n)$. If $d$ is a left-invariant metric on $\mathrm{Sym}(\mathbb{N})$, compatible with the topology of pointwise convergence, then $\bigl(\mathrm{Sym}(n),d\!\!\upharpoonright_{\mathrm{Sym}(n)},\mu_{n}\bigr)_{n \in \mathbb{N}}$ dissipates, thus fails to admit a $d_{\mathrm{conc}}$-Cauchy subsequence. \end{cor}

Corollary~\ref{corollary:final} is to be compared with the following well-known result due to Maurey~\cite{maurey} (see also~\cite{MilmanSchechtman,PestovBook}): with regard to the \emph{normalized Hamming distances} \begin{displaymath}
	d_{\mathrm{Ham},n} (g,h) \defeq \frac{\vert \{ i \in \{ 1,\ldots,n\} \mid g(i) \ne h(i ) \}\vert}{n} \qquad (g,h \in \mathrm{Sym}(n))
\end{displaymath} and the normalized counting measures $\mu_{n}$ on $\mathrm{Sym}(n)$, the sequence $(\mathrm{Sym}(n),d_{\mathrm{Ham},n},\mu_{n})_{n\geq 1}$ constitutes a normal L\'evy family, thus concentrates to a singleton space.

Pestov's Problem~\ref{problem} has an interesting sibling, namely~\cite[Conjecture~7.4.26]{PestovBook}, which has been confirmed recently in~\cite{EquivariantConcentration} as part of the following more general result.

\begin{thm}[\cite{EquivariantConcentration}, Theorem~1.1]\label{theorem:equivariant.concentration} Let $G$ be a second-countable topological group equipped with a right-invariant compatible metric $d$. Suppose that there exists a sequence $(\mu_{n})_{n \in \mathbb{N}}$ of Borel probability measures on $G$ with compact supports $K_{n} \defeq \spt \mu_{n}$ $(n \in \mathbb{N})$ such that \begin{itemize}[leftmargin=11mm]
	\item[$(1)$] $(\mu_{n})_{n \in \mathbb{N}}$ UEB-converges to invariance over $G$, and \smallskip
	\item[$(2)$] $\left(K_{n}, d\!\!\upharpoonright_{K_{n}}, \mu_{n}\!\!\upharpoonright_{K_{n}}\right)_{n \in \mathbb{N}}$ concentrates to a fully supported, compact $mm$-space $(X,d_{X},\mu_{X})$.
\end{itemize} Then there exists a topological embedding $\psi \colon X \to \mathrm{S}(G)$ such that the push-forward measure $\psi_{\ast}(\mu_{X})$ is $G$-invariant. In particular, $\psi (X)$ is a $G$-invariant subspace of $\mathrm{S}(G)$. \end{thm}

Combining Theorem~\ref{theorem:equivariant.dissipation} with the results of~\cite{EquivariantConcentration}, we subsequently deduce a dichotomy between concentration and dissipation in the context of non-archimedean second-countable topological groups, that is, Corollary~\ref{corollary:dichotomy}. This dichotomy will distinguish precompact topological groups from non-precompact ones. Preparing the statement of Corollary~\ref{corollary:dichotomy}, let us briefly clarify some notions. For a topological group $G$, we consider its \emph{Bohr compactification} $\kappa_{G} \colon G \to \kappa G$ (see~\cite{holm,deVries}),~i.e., $\kappa G$ is the Gelfand spectrum of the $C^{\ast}$-algebra $\mathrm{AP}(G)$ of all almost periodic continuous bounded complex-valued functions on $G$ equipped with the continuous group structure given by \begin{displaymath}
	(\mu \nu)(f) \defeq \mu (g \mapsto \nu (f \circ \lambda_{g})) \qquad (\mu, \nu \in \kappa G, \, f \in \mathrm{AP}(G)) ,
\end{displaymath} and $\kappa_{G} \colon G \to \kappa G$ is the continuous homomorphism defined by \begin{displaymath}
	\kappa_{G}(x)(f) \defeq f(x) \qquad (x \in G, \, f \in \mathrm{AP}(G)) ,
\end{displaymath} which has dense image in $\kappa G$. It is is well known that a topological group $G$ is precompact if and only if $\kappa_{G}$ is a topological embedding. This fact readily implies that, if $G$ is a metrizable precompact topological group and $d$ is any right-invariant compatible metric on~$G$, then there exists a uniquely determined -- necessarily right-invariant -- compatible metric $d_{\kappa G}$ on $\kappa G$ such that $d_{\kappa G}(\kappa_{G}(x),\kappa_{G}(y)) = d (x,y)$ for all $x,y \in G$ (see e.g.~\cite[Lemma~4.5]{EquivariantConcentration}). 

\begin{cor}\label{corollary:dichotomy} Let $G$ be a non-archimedean second-countable topological group together with a right-invariant compatible metric $d$ and let $(\mu_{n})_{n \in \mathbb{N}}$ be a sequence of Borel probability measures on $G$ with compact supports $K_{n} \defeq \spt \mu_{n}$ $(n \in \mathbb{N})$. Suppose that $(\mu_{n})_{n \in \mathbb{N}}$ UEB-converges to invariance over $G$. Then, either \begin{enumerate}[leftmargin=11mm]
	\item[$(1)$] $G$ is precompact, and then $\left(K_{n}, d\!\!\upharpoonright_{K_{n}}, \mu_{n}\!\!\upharpoonright_{K_{n}}\right)_{n \in \mathbb{N}}$ concentrates to $(\kappa G,d_{\kappa G},\mu_{\kappa G})$, or \smallskip
	\item[$(2)$] $G$ is not precompact, and then $\left(K_{n}, d\!\!\upharpoonright_{K_{n}}, \mu_{n}\!\!\upharpoonright_{K_{n}}\right)_{n \in \mathbb{N}}$ dissipates.
\end{enumerate} \end{cor}

\begin{proof} The first assertion is proved in~\cite{EquivariantConcentration}: if $G$ is precompact, then $\left(K_{n}, d\!\!\upharpoonright_{K_{n}}, \mu_{n}\!\!\upharpoonright_{K_{n}}\right)_{n \in \mathbb{N}}$ concentrates to $(\kappa G,d_{\kappa G},\mu_{\kappa G})$ (see~\cite[Proof of Theorem~1.1, first case, pages~10--11]{EquivariantConcentration}). If $G$ is not precompact, then $G$, being non-archimedean, must admit an open subgroup of infinite index, in which case the desired conclusion follows by Theorem~\ref{theorem:equivariant.dissipation} and Proposition~\ref{proposition:uniform.dissipation}. \end{proof}

In particular, Corollary~\ref{corollary:dichotomy} substantiates that the only manifestations of the kind of concentration described by Theorem~\ref{theorem:equivariant.concentration} to be found in the realm of non-archimedean groups are those occuring -- in a trivial fashion -- in precompact topological groups.

Extending earlier work of Pestov~\cite{pestov10} as well as Glasner, Tsirelson and Weiss~\cite{GlasnerTsirelsonWeiss}, it was shown in~\cite[Theorem~3.9]{PestovSchneider} that if $G$ is any topological group admitting a net of Borel probability measures that concentrates in $G$ (see~\cite[Definition~2.6]{pestov02}) and, at the same time, UEB-converges to invariance over $G$, then $G$ is \emph{whirly amenable}, i.e., \begin{itemize}
	\item[$\bullet$] \, $G$ is amenable, and 
	\item[$\bullet$] \, every invariant regular Borel probability measure on a compact $G$-space is supported on the set of fixed points,
\end{itemize} which particularly entails that no non-trivial compact $G$-space can possibly admit an ergodic regular Borel probability measure. In the light of these results, one might be tempted to conjecture that the mere existence of \emph{some} net of Borel probability measures on a topological group, simultaneously dissipating and UEB-converging to invariance, would have interesting dynamical or ergodic-theoretical consequences, beyond the obvious non-precompactness and amenability. We will finish the present note by disposing of this hope: in fact, every non-precompact amenable topological group admits such a net of measures (Proposition~\ref{proposition:precompact}).

Recall that a topological group $G$ is said to be \emph{precompact} if $G$ is precompact with respect to its right uniformity, that is, for every $U \in \mathcal{U}(G)$ there exists a finite subset $F \subseteq G$ such that $G = UF$. The following characterization of precompact groups was found independently by Uspenskij (unpublished, cf.~a footnote in~\cite{uspenskij}) and Solecki~\cite{solecki}. For a short proof, the reader is referred to~\cite[Proposition~4.3]{BouziadTroallic}.

\begin{lem}\label{lemma:usp} A topological group $G$ is precompact if and only if, for every $U \in \mathcal{U}(G)$, there exists a finite subset $F \subseteq G$ with $G = FUF$. \end{lem}

A proof of the following easy consequence of Lemma~\ref{lemma:usp} may be found in~\cite{EquivariantConcentration}.

\begin{cor}[see Corollary~4.3 in~\cite{EquivariantConcentration}]\label{corollary:usp} Let $G$ be a non-precompact topological group. Then there exists some $U \in \mathcal{U}(G)$ such that, for every sequence $(F_{n})_{n \in \mathbb{N}}$ of finite subsets of~$G$, there exists $(g_{n})_{n \in \mathbb{N}} \in G^{\mathbb{N}}$ such that \begin{displaymath}
	\forall m,n \in \mathbb{N}, \, m \ne n \colon \qquad UF_{m}g_{m} \cap UF_{n}g_{n} \, = \, \emptyset .
\end{displaymath} \end{cor}

We may now prove the aforementioned result.

\begin{prop}\label{proposition:precompact} Let $G$ be a topological group. The following are equivalent. \begin{enumerate}
	\item[$(1)$] $G$ is amenable and not precompact. \smallskip
	\item[$(2)$] $G$ admits a net of Borel probability measures simultaneously dissipating in $G$ and UEB-converging to invariance over $G$.\smallskip
	\item[$(3)$] $G$ admits a net of finitely supported regular Borel probability measures simultaneously dissipating in $G$ and UEB-converging to invariance over $G$.
\end{enumerate} \end{prop}

\begin{proof} (2)$\Longrightarrow$(1). This is due to Theorem~\ref{theorem:topological.day} and Proposition~\ref{proposition:dissipation.in.uniform.spaces}. 
	
(3)$\Longrightarrow$(2). Trivial.
	
(1)$\Longrightarrow$(3). Suppose that $G$ is not precompact. Let $U \in \mathcal{U}(G)$ be as in Corollary~\ref{corollary:usp}. According to Theorem~\ref{theorem:topological.day}, there exists a net $(\mu_{i})_{i \in I}$ of finitely supported regular Borel probability measures on $G$ UEB-converging to invariance over $G$. Let $S_{i} \defeq \spt \mu_{i}$ for all $i \in I$. Consider the directed set $J \defeq I \times (\mathbb{N} \setminus \{ 0 \})$ endowed with the partial order $\leq_{J}$ given by \begin{displaymath}
	(i_{0},n_{0}) \, \leq_{J} \, (i_{1},n_{1}) \quad :\Longleftrightarrow \quad i_{0} \leq_{I} i_{1}, \ n_{0} \leq n_{1} \qquad ((i_{0},n_{0}), (i_{1},n_{1}) \in J) .
\end{displaymath} Due to Corollary~\ref{corollary:usp}, for each $(i,n) \in J$ there exists a subset $F_{i,n} \subseteq G$ such that \begin{itemize}
	\item[(i)] \, $\vert F_{i,n} \vert = n$, and
	\item[(ii)] \, $US_{i}g \cap US_{i}h = \emptyset$ for any two distinct $g,h \in F_{i,n}$.
\end{itemize} For every $(i,n) \in J$, let us consider the regular Borel probability measure $\nu_{i,n} \defeq \mu_{i} \ast \delta_{F_{i,n}}$ on~$G$, and note that $\spt \nu_{i,n} = \bigcupdot_{g \in F_{i,n}} S_{i}g$ is finite. Since the net $(\mu_{i})_{(i,n) \in J}$ UEB-converges to invariance over~$G$, so does $(\nu_{i,n})_{(i,n) \in J}$ according to Lemma~\ref{lemma:convoluted.convergence}. Therefore, it only remains to argue that $(\nu_{i,n})_{(i,n) \in J}$ $U_{\Rsh}$-dissipates in $G$. Indeed, if $(i,n) \in J$, then $\mathcal{B}_{i,n} \defeq \{ S_{i}g \mid g \in F_{i,n} \}$ is a finite collection of finite (thus Borel) subsets of $G$, and moreover \begin{itemize}
	\item[$\bullet$] \, $(B \times C) \cap U_{\Rsh} = \emptyset$ for any two distinct $B,C \in \mathcal{B}_{i,n}$ by~(ii),
	\item[$\bullet$] \, $\nu_{i}\!\left( \bigcup \mathcal{B}_{i,n} \right) = \nu_{i}\!\left( S_{i}F_{i,n} \right) = 1$ as $\spt \nu_{i,n} = S_{i}F_{i,n}$,
	\item[$\bullet$] \, $\nu_{i}(S_{i}g) = \tfrac{1}{n}$ for every $g \in F_{i,n}$ by~(i) and~(ii).
\end{itemize} This completes the proof. \end{proof}

We conclude with an additional remark concerning the proof of Proposition~\ref{proposition:precompact}.

\begin{remark} Let $G$ be a topological group and $(\mu_{i})_{i \in I}$ be a net of finitely supported regular Borel probability measures on $G$ UEB-converging to invariance. Let $S_{i} \defeq \spt \mu_{i}$ for $i \in I$.
	
(1) If $G$ is infinite, then $\vert S_{i} \vert \longrightarrow \infty$ as $i \to I$. To see this, let $m \in \mathbb{N}$. As $G$ is infinite, we find a finite subset $E \subseteq G$ with $\vert E \vert > m^{2}$. Pick $U \in \mathcal{U}(G)$ so that $\left( UU^{-1} \right) \cap \left(E^{-1}E\right) = \{ e \}$. By Urysohn's lemma for uniform spaces, there exists $f \in \mathrm{RUCB}(G)$ such that $f(G) \subseteq [0,1]$, $f(e) = 1$ and $f(x) = 0$ whenever $x \in G\setminus U$. For every subset $S \subseteq G$, define $f_{S} \colon G \to [0,1]$ by \begin{displaymath}
	f_{S}(x) \defeq \sup\nolimits_{s \in S} f\! \left(xs^{-1}\right) \qquad (x \in G) .
\end{displaymath} It is not difficult to see that $\{ f_{S} \mid S \subseteq G \} \in \mathrm{RUEB}(G)$. Since the net $(\mu_{i})_{i \in I}$ UEB-converges to invariance over $G$, there exists $i_{0} \in I$ such that \begin{displaymath}
	\forall i \in I, \, i \geq i_{0} \colon \quad \sup\nolimits_{g \in E} \sup\nolimits_{S \subseteq G} \left\lvert \int f_{S} \, d\mu_{i} - \int f_{S} \circ \lambda_{g^{-1}} \, d\mu_{i} \right\rvert \, \leq \, \tfrac{1}{2} .
\end{displaymath} We show that $\vert S_{i} \vert > m$ for all $i \in I$ with $i \geq i_{0}$. To this end, let $i \in I$ with $i \geq i_{0}$. Note that $\int f_{S_{i}} \, d\mu_{i} = 1$, because $\mu_{i}(S_{i}) = 1$ and $f_{S_{i}}(S_{i}) = \{ 1 \}$. Hence, for each $g \in E$, the above implies that $\int f_{S_{i}} \circ \lambda_{g^{-1}} \, d\mu_{i} \geq \tfrac{1}{2}$, and so $\left(f_{S_{i}} \circ \lambda_{g^{-1}}\right)\vert_{S_{i}} \neq \mathbf{0}$, which entails that $gUS_{i} \cap S_{i} \ne \emptyset$,~i.e., $g \in S_{i}S_{i}^{-1}U^{-1}$. Since $\left( UU^{-1} \right) \cap \left(E^{-1}E\right) = \{ e \}$, we conclude that $\vert E \vert \leq \left\lvert S_{i}S_{i}^{-1} \right\rvert \leq \vert S_{i} \vert^{2}$, and therefore $m < \vert S_{i} \vert$ as desired.

(2) Let $n_{i} \defeq \vert S_{i} \vert$ for $i \in I$. Suppose that $G$ is not precompact. Then $n_{i} \longrightarrow \infty$ $(i \to I)$ by~(1). Moreover, by Corollary~\ref{corollary:usp}, for each $i \in I$ there exists an $n_{i}$-element subset $F_{i} \subseteq G$ such that $US_{i}g \cap US_{i}h = \emptyset$ for any two distinct elements $g,h \in F_{i}$. An argument analogous ot the one given in the proof of Proposition~\ref{proposition:precompact} now shows that $(\mu_{i} \ast \delta_{F_{i}})_{i \in I}$ constitutes a net of finitely supported regular Borel probability measures on $G$, simultaneously dissipating in and UEB-converging to invariance over $G$. \end{remark}

\section*{Acknowledgments}

This research has been supported by funding of the Excellence Initiative by the German Federal and State Governments as well as the Brazilian Conselho Nacional de Desenvolvimento Cient\'{i}fico e Tecnol\'{o}gico (CNPq), processo 150929/2017-0. The author is deeply indebted to Vladimir Pestov for a number of inspiring and insightful discussions about the concentration of measured metric spaces, as well as to Tom Hanika for a helpful exchange on combinatorics of finite permutation groups. Also, the kind hospitality of CFM--UFSC (Florian\'opolis) during the origination of this work is gratefully acknowledged. Furthermore, the author would like to express his sincere gratitude towards the anonymous referee for their careful reading and valuable remarks, including a substantial simplification of an earlier proof of Theorem~\ref{theorem:equivariant.dissipation}!


\end{document}